\documentclass[11pt,table]{amsart}
\usepackage{amsmath}
\usepackage[margin=1.0in]{geometry}
\usepackage{amscd,amsmath,amsxtra,amsthm,amssymb,stmaryrd,xr,mathrsfs,mathtools,enumerate,commath, comment, enumitem, graphicx}
\usepackage{amsfonts, amssymb, amsthm, amsmath, braket, hyperref, mathtools, comment, mathrsfs, enumitem, cleveref}
\usepackage{stmaryrd}
\usepackage{orcidlink}
\usepackage{multirow}
\usepackage{xcolor}
\usepackage{commath}
\usepackage{comment}
\usepackage{algorithm}
\usepackage{algorithmic}
\usepackage{tikz-cd}
\usepackage{tkz-graph}
\usepackage{colonequals}
\usepackage{hyperref}
\usepackage{graphicx}
\usepackage{bigstrut}
\usepackage{longtable, multirow, array}
\usepackage{enumitem}
\usepackage{placeins}

\usepackage{longtable} 
\usepackage{pdflscape} 
\usepackage{booktabs}
\usepackage{hyperref}
\definecolor{vegasgold}{rgb}{0.77, 0.7, 0.35}
\definecolor{darkgoldenrod}{rgb}{0.72, 0.53, 0.04}
\definecolor{gold(metallic)}{rgb}{0.83, 0.69, 0.22}
\hypersetup{
	colorlinks=true,
	linkcolor=darkgoldenrod,
	filecolor=brown,      
	urlcolor=gold(metallic),
	citecolor=darkgoldenrod,
}

\usepackage[all,cmtip]{xy}
\DeclareFontFamily{U}{wncy}{}
\DeclareFontShape{U}{wncy}{m}{n}{<->wncyr10}{}
\DeclareSymbolFont{mcy}{U}{wncy}{m}{n}
\DeclareMathSymbol{\Sh}{\mathord}{mcy}{"58}
\usepackage[T2A,T1]{fontenc}
\usepackage[OT2,T1]{fontenc}

\usepackage{tikz}
\usetikzlibrary{shapes.geometric}
\usetikzlibrary{decorations.markings}
\tikzset{every loop/.style={min distance=10mm,looseness=10}}
\tikzstyle{vertex}=[auto=left,circle,minimum size=1pt,inner sep=0pt]

\newtheorem{theorem}{Theorem}[section]
\newtheorem{lemma}[theorem]{Lemma}

\newtheorem*{theorem*}{Theorem}
\newtheorem*{ass*}{Assumption}
\newtheorem{definition}[theorem]{Definition}
\newtheorem{corollary}[theorem]{Corollary}
\newtheorem{remark}[theorem]{Remark}
\newtheorem{example}[theorem]{Example}

\newtheorem{proposition}[theorem]{Proposition}

\newcommand{\Q}{\mathbb{Q}}

\newcommand{\D}{\mathfrak{D}}

\numberwithin{equation}{section}

\begin{document}

\title[A Deterministic Cryptographic Prime Generation Chain]{A Deterministic Cryptographic Prime Generation Chain over Monogenic Cubic Number Fields and their Generalizations}

\author[A. Jakhar]{Anuj Jakhar\, \orcidlink{0009-0007-5951-2261}}
\author[R.~Kalwaniya]{Ravi Kalwaniya\, \orcidlink{0009-0008-6964-5276}}
\address[Jakhar, Kalwaniya]{Department of Mathematics, Indian Institute of Technology Madras, Chennai, Tamil Nadu, India-600036}
\email{anujjakhar@iitm.ac.in\,   \,ravikalwaniya3@gmail.com}

\subjclass[2020]{11A51, 11Y11, 11R21, 94A60}
\keywords{Primality testing, prime generation chain, pure prime degree fields, generalized unitary group, deterministic algorithms.}

\begin{abstract}
Generating primes is a fundamental problem in modern cryptography. Deterministic primality tests work well for special integers such as Mersenne or Proth primes, but these forms are quite restrictive. In this paper, we give a direct method to construct new primes from known ones. Starting with a seed prime $q \equiv 1 \pmod{3}$, we construct an integer $N \equiv 1 \pmod{3}$ satisfying $(2N + 1)^2 \equiv -3 \pmod{q}$. We then prove that $N$ is prime using the structure of monogenic pure cubic fields $K = \mathbb{Q}(\sqrt[3]{d})$. The resulting test requires only a single modular exponentiation and runs in $\tilde{\mathcal{O}}(\log^2 N)$ time. Finally, we show how this construction extends to pure number fields of arbitrary prime degree.
\end{abstract}
\maketitle
\section{Introduction}\label{intro}
The generation of large prime numbers plays an important role in modern cryptography. In practice, fast tests such as the Miller--Rabin test \cite{MR480295,MR566880} are widely used to quickly eliminate composite numbers, but only indicate that a number is likely to be prime. To obtain a rigorous proof of primality, deterministic methods are required, as shown by Agrawal, Kayal, and Saxena \cite{MR2123939}. This motivates the search for methods that are efficient in practice and mathematically rigorous.

The classical methods of Lucas \cite{MR1505161}, Proth \cite{Proth1878}, and Lehmer \cite{MR1502953} established unconditional tests for numbers such as $2^n-1$ and $k \cdot 2^n+1$. For a brief history of primality testing and how the tests of Lucas fit into this broader historical context, see~\cite{MR2598366}. Later, researchers worked to generalize these criteria for larger families of integers, especially those of the form $h \cdot p^n \pm 1$. In a foundational series of papers, Williams and his co-authors~\cite{MR583519,MR439719,MR3403162,MR311559,MR638738,MR866123,MR476625,MR396390,MR537980,MR314747} developed rigorous primality tests for a variety of bases. Subsequently, Bosma \cite{MR1197510,MR2076210}  and Berrizbeitia et al. \cite{MR1487359, MR2047101,MR1862113,MR2630012,MR2407072} applied explicit reciprocity laws to similar structures. 

This direction remains highly active today, and many researchers continue to develop rigorous optimized tests for specialized families such as $h \cdot 2^n \pm 1$ \cite{MR3464611}, $4Kp^n - 1$ \cite{Grau2020}, $Kp^n+1$ \cite{Grau2015}, $2kp^m-1$ \cite{MR2254741},  and $(p-1)p^n-1$ \cite{MR1697651} . Recent work has further refined these explicit criteria and improved their algorithmic efficiency~(see, \cite{MR4372238,Jakhar2026,MR3941463,MR4467005,MR4888430}). However, from a cryptographic point of view, these structured numbers have a distinct drawback. If a large prime power divides $N \pm 1$, they often show predictable bit patterns, making them vulnerable to specialized factoring attacks. In practice, cryptography prefers primes that look completely random, but still admit a deterministic proof.

To avoid the difficulty of factoring random numbers, cryptographers have long used a constructive ``bottom-up'' approach. Instead of generating a random candidate and trying to factor $N-1$, classical algorithms by Shawe-Taylor \cite{ShaweTaylor1986} and Maurer \cite{Maurer1995} start with a known seed prime $q$ and build $N$ directly around it using Pocklington's theorem. Later research extended these constructive techniques to the quadratic case $N+1$ using Lucas sequences. Although highly efficient, these linear and quadratic chains still leave the resulting primes with rigid structures $N \pm 1$.

Researchers have also studied higher-degree extensions. Recently, Baillie, Fiori, and Wagstaff \cite{MR4273120} refined the Baillie-PSW test by adding Lucas--V pseudoprime conditions. In another direction, Ramzy \cite{Ramzy2023} extended Proth's theorem and introduced $a$-Safe primes, giving new ways to generate primes. At the same time, cubic structures are being explored more actively. In 2021, Roettger and Williams \cite{MR4316955} successfully moved classical Lucas functions into the cubic domain, proving that third-order linear recurrence sequences $(C_n, W_n)$ could verify candidate integers where $p^n \mid (N^2+N+1)$. Furthermore, Di Domenico and Murru \cite{MR4893469} developed a highly performant primality test that cleverly utilizes the projectivization of Pell's cubic alongside third-order linear recurrence sequences. For a comprehensive, modern treatment of Lucas sequences, see the recent text by Ballot and Williams \cite{MR4676377}.

Even with these improvements, a difficult computational barrier remains. In these cubic tests, the required factor $q$ actually has to be larger than $N$ itself. Therefore, verifying an arbitrary number $N$ requires first choosing a candidate and subsequently performing a computationally expensive  search for massive prime factors just to prove it is prime. Even Elliptic Curve Primality Proving (ECPP) \cite{MR1199989}, which circumvents these strict algebraic requirements, still relies on a step-by-step search that takes $\tilde{\mathcal{O}}(\log^3 N)$ time. For a comprehensive overview of deterministic primality proving and recent developments in this direction, we refer the reader to the foundational texts by Crandall and Pomerance \cite{MR2156291}, Wagstaff \cite{MR3135977}, and the recent monograph by Roettger and Williams \cite{MR4944565}.

In this paper, we take a different approach to avoid this difficulty. Working initially in pure cubic fields, we offer an alternative that naturally resists standard $N \pm 1$ factorization attacks while keeping the generation process highly efficient. Instead of picking a random integer $N$ and trying to prove that it is prime, we start with a known seed prime $q \equiv 1 \pmod{3}$ and construct $N$ directly. More precisely, we work in cubic fields of the form $K = \mathbb{Q}(\sqrt[3]{d})$ and define $N$ using the congruence
\begin{equation}
(2N+1)^2 \equiv -3 \pmod q.
\end{equation}
This ensures that the expression $N^2+N+1$ inherently takes the form 
\[
N^2+N+1 = k \cdot q,
\]
where $k$ is some non zero integer. At the same time, the way $N$ is built ensures that $q$ is exceptionally large. In fact, since $q \approx N^2$, the structural condition $q > N^{3/2}-1$ is automatically satisfied. This completely removes the need to search for a large prime factor afterwards. We then work in the cubic quotient ring $\mathcal{R}_N = \mathcal{O}_K/N\mathcal{O}_K$ (where $\mathcal{O}_K$ is the ring of integers of $K$) and consider the generalized unitary group $\mathcal{G}_N(d)$. The precise definitions and properties of these algebraic structures are developed in Section~\ref{sec:2}. This setting allows us to isolate the cyclotomic part and produce a Lucas--type certificate along with the number $N$. All computations are carried out using polynomial arithmetic modulo $x^3-d$, keeping the procedure very fast. In particular, the verification step reduces to just a few modular exponentiations, leading to a total bit complexity of $\tilde{\mathcal{O}}(\log^2 N)$. Further, we extend results for monogenic pure prime degree fields.\\[1mm]
\textbf{Our main contributions are summarized as follows : }First, we introduce a deterministic, bottom-up prime generation chain that constructs a candidate prime $N$ from a known seed prime $q \equiv 1 \pmod 3$ via the relation $(2N+1)^2 \equiv -3 \pmod q$. By working within the quotient ring of monogenic pure cubic fields $K = \mathbb{Q}(\sqrt[3]{d})$, we establish an explicit primality certificate using the generalized unitary group $\mathcal{G}_N(d)$ and the cyclotomic polynomial $\Phi_q(X)$. This constructive approach inherently satisfies the necessary structural bounds without heavy factorization steps, allowing the primality of $N$ to be rigorously verified in $\tilde{\mathcal{O}}(\log^2 N)$ time. Finally, we fully generalize this algebraic framework and its corresponding primality certificate to pure number fields of arbitrary prime degree $p$.\\[1mm]
\textbf{Structure of the Paper :} The remainder of this paper is organized as follows. In Section \ref{sec:2}, we establish the algebraic foundations, detailing the Kummer-Dedekind factorization and the structural properties of the generalized unitary group $\mathcal{G}_N(d)$. Section \ref{sec:3} introduces the core cryptographic prime generation chain, presenting the construction of the candidate prime $N$ and explicitly proving the bounding condition $q > N^{3/2}-1$. In Section \ref{sec:4}, we develop the general theoretical framework for the cubic Lucasian primality certificate. In Section \ref{sec:5}, we apply this theory to derive an efficient deterministic primality certificate for our target integers. Section \ref{sec:6} formalizes the complete generation algorithm and establishes its strict computational bit complexity $\tilde{\mathcal{O}}(\log^2 N)$. To optimize practical implementations, Section \ref{sec:7} provides a fast probabilistic filter using Cubic Strong Probable Primes. In Section \ref{sec:8}, we present computational experiments that directly support our theoretical complexity bounds. Building upon this foundation, Section \ref{sec:9} expands our methodology, fully generalizing the algebraic framework and structural primality certificate to pure number fields of arbitrary prime degree $p$. Finally, Section \ref{sec:10} provides the concluding remarks.
\section{Algebraic Foundations }\label{sec:2}

The efficiency of an algebraic primality test depends on the arithmetic of the ring of integers $\mathcal{O}_K$. In our approach, the main step is modular exponentiation in the quotient ring $\mathcal{R}_N = \mathcal{O}_K / N\mathcal{O}_K$. When the integral basis\footnote{An integral basis of a number field $K$ of degree $n$ is a set of $n$ algebraic integers $\{b_1, b_2, \dots, b_n\}$ in $\mathcal{O}_K$ such that every element in $\mathcal{O}_K$ can be expressed uniquely as a linear combination $\sum_{i=1}^n a_i b_i$ with coefficients $a_i \in \mathbb{Z}$.} of $\mathcal{O}_K$ includes fractional elements, the computations in $\mathcal{R}_N$ become more involved. Then one must keep track of the denominators and compute modular inverses. This additional complication increases the computational cost and does not allow us to achieve the desired $\tilde{\mathcal{O}}(\log^2 N)$ complexity. To keep arithmetic manageable, we therefore restrict our attention to number fields whose rings of integers admit a simple polynomial representation, which naturally leads to monogenic fields. This is defined as follows:
\begin{definition}
A number field $K$ of degree $n$ is said to be \textit{monogenic} if its ring of integers $\mathcal{O}_K$ is generated by a single element $\alpha \in \mathcal{O}_K$. That is, $\mathcal{O}_K = \mathbb{Z}[\alpha]$, and the set $\{1, \alpha, \alpha^2, \dots, \alpha^{n-1}\}$ forms an integral basis for $\mathcal{O}_K$, known as a \textit{power basis}.
\end{definition}
From an algorithmic perspective, working within a monogenic field is ideal. When $\mathcal{O}_K = \mathbb{Z}[\alpha]$, any element in the quotient ring $\mathcal{R}_N$ can be uniquely represented as a polynomial in $\alpha$ of degree less than $n$, with coefficients in $\mathbb{Z}/N\mathbb{Z}$. Multiplication is simply polynomial multiplication reduced modulo the minimal polynomial of $\alpha$ and modulo $N$, completely avoiding fractional arithmetic.
We now apply this framework to pure cubic fields. Let $K = \mathbb{Q}(\theta),$ where $\theta^3 = d$. Without loss of generality, we may assume $d = ab^2$, where $a$ and $b$ are relatively prime squarefree integers. We denote the ring of integers of $K$ by $\mathcal{O}_K$, and the absolute discriminant of the field by $d_K$. The minimal polynomial of $\theta$ is $f(x) = x^3 - d$, which has discriminant $\D_f = -27d^2$. It is a fundamental result in algebraic number theory that the polynomial discriminant and the field discriminant $d_K$, are related by the index equation:
\begin{equation}\label{formula}
\D_f = [\mathcal{O}_K : \mathbb{Z}[\theta]]^2 d_K.
\end{equation}
To understand when $K$ is monogenic with respect to $\theta$, we must examine the general integral basis of $\mathcal{O}_K$. A classical theorem by Dedekind classifies this basis depending on the prime factors of $d$ and its congruence class modulo 9.
\begin{theorem}
\label{thm:integral_basis}
Let $K = \mathbb{Q}(\theta)$ be a cubic field with $\theta^3 = d = ab^2$ as defined above. Then the following holds:
\begin{enumerate}
    \item[\textup{(i)}] If $d \not\equiv 1, 8 \pmod 9$, then $\{1, \theta, \theta^2/b\}$ is an integral basis of $\mathcal{O}_K$ and $d_K = -27a^2b^2$.
    \item[\textup{(ii)}] If $d \equiv 1 \pmod 9$, then $\{\theta, \theta^2/b, (1 + \theta + \theta^2)/3\}$ is an integral basis of $\mathcal{O}_K$ and $d_K = -3a^2b^2$.
    \item[\textup{(iii)}] If $d \equiv 8 \pmod 9$, then $\{\theta, \theta^2/b, (1 - \theta + \theta^2)/3\}$ is an integral basis of $\mathcal{O}_K$ and $d_K = -3a^2b^2$.
\end{enumerate}
\end{theorem}

\noindent From Theorem \ref{thm:integral_basis}, it is clear that the general integral basis contains denominators of $b$, and depending on the congruence of $d$ denominators of $3$. Since our goal is to compute using the simple power basis $\{1, \theta, \theta^2\}$, we must restrict our choice of $d$ so that these fractional terms disappear entirely. The following corollary establishes the exact conditions required to guarantee this.

\begin{corollary}\label{cor:monogenic}
Let $K = \mathbb{Q}(\theta)$ with $\theta^3 = d$. The ring of integers is  $\mathcal{O}_K = \mathbb{Z}[\theta]$ if and only if $d$ is a squarefree integer and $d \not\equiv 1, 8 \pmod 9$.
\end{corollary}
\noindent By restricting the field parameter $d$ to squarefree integers with $d \not\equiv \pm 1 \pmod 9$, we ensure that field $K$ is monogenic with a power basis generator $\theta$. In this case, the discriminant of field is given by $d_K = -27d^2$. This restriction allows the primality test to run efficiently, since all computations in $\mathcal{R}_N$ can be performed using modular polynomial arithmetic over $\mathbb{Z}/N\mathbb{Z}$.\\
Instead of working with standard integers modulo $N$, our primality test operates inside the quotient ring $\mathcal{R}_N = \mathcal{O}_K / N\mathcal{O}_K$. To understand the mathematical structure of this ring, we determine how the ideal generated by $N$ factors inside the ring of integers $\mathcal{O}_K$. To do this, we rely on a foundational result from algebraic number theory known as the Kummer-Dedekind Theorem (see, \cite[Theorem 4.8.13]{Cohen1993}),  which is as follows:

\begin{theorem}\label{thm:kummer_dedekind}
Let $K = \mathbb{Q}(\theta)$ be a number field where $\theta$ is an algebraic integer with minimal polynomial $f(x) \in \mathbb{Z}[x]$. Let $p$ be a prime number. If $p$ does not divide the index $[\mathcal{O}_K : \mathbb{Z}[\theta]]$, then the prime ideal factorization of $p\mathcal{O}_K$ is determined by the factorization of $f(x)$ modulo $p$. Specifically, if $f(x)$ is irreducible modulo $p$, then $p\mathcal{O}_K$ remains a prime ideal in $\mathcal{O}_K$.
\end{theorem}
\noindent When $N\mathcal{O}_K$ is a prime ideal, the quotient ring $\mathcal{R}_N = \mathcal{O}_K / N\mathcal{O}_K$ becomes a finite field. In this field extension, the Galois group is generated by the Frobenius automorphism, which maps any element $x \mapsto x^N$. This allows us to formally define the relative norm map and the generalized unitary group.

\begin{definition}
Suppose that $\mathcal{R}_N \cong \mathbb{F}_{N^3}$. We define the norm map $\mathcal{N} : \mathcal{R}_N^\times \to \mathbb{F}_N^\times$ by $\mathcal{N}(w) = w^{1+N+N^2}.$ We define the generalized unitary group as follows:
\[
\mathcal{G}_N(d) := \{ w \in \mathcal{R}_N^\times \mid w^{N^2+N+1} \equiv 1 \pmod N \}.
\]
\end{definition}
\begin{proposition}[Euler's Criterion for Cubic Residues]
\label{thm:euler_cubic}
Let $p$ be a prime such that $p \equiv 1 \pmod 3$, and  $a$ be an integer such that $\gcd(a, p) = 1$. Then $a$ is a cubic residue modulo $p$ if and  only if $a^{\frac{p-1}{3}} \equiv 1 \pmod p.$
\end{proposition}
\noindent We omit the  proof for the cubic case here, as the fully generalized version of this criterion for an arbitrary prime degree is  established later in Section \ref{sec:9}. As an immediate consequence of this criterion, if $a^{\frac{p-1}{3}} \not\equiv 1 \pmod p$, then $a$ is a cubic non-residue, which means that the polynomial $X^3 - a$ is irreducible over the finite field $\mathbb{F}_p$.\\
\noindent We now show that, by imposing suitable conditions on $N$ and $d$, the quotient ring becomes a finite field, allowing us to determine the exact size of the generalized unitary group.

\begin{lemma}\label{lem:ring_structure}
Let $N$ be a prime such that $N \equiv 1 \pmod{3}$ and $\gcd(N,\,3d)=1$, where $d \in \mathbb{Z}$ is not a perfect cube. Suppose that $d^{\frac{N-1}{3}} \not\equiv 1 \pmod{N}.$
Then the following statements hold:
\begin{enumerate}
    \item[\textup{(i)}] The polynomial $X^3 - d$ is irreducible over $\mathbb{F}_N$.
    
    \item[\textup{(ii)}] The ideal $N\mathcal{O}_K$ is prime. In particular, $\mathcal{R}_N := \mathcal{O}_K / N\mathcal{O}_K \cong \mathbb{F}_{N^3}.$
    \item[\textup{(iii)}] The group $\mathcal{G}_N(d) := \{ w \in \mathcal{R}_N^\times \mid w^{N^2+N+1} = 1 \}$
    is cyclic of order $N^2+N+1$.
\end{enumerate}
\end{lemma}

\begin{proof}
(i). Since $N \equiv 1 \pmod{3}$, the group $\mathbb{F}_N^\times$ contains a primitive third root of unity. It follows that a polynomial of the form $X^3 - d$ is reducible over $\mathbb{F}_N$ if and only if $d$ is a cube in $\mathbb{F}_N^\times$. By Euler's criterion for cubic residues, this holds if and only if
\[
d^{\frac{N-1}{3}} \equiv 1 \pmod{N}.
\]
By hypothesis, this is not possible. Therefore, $X^3 - d$ is irreducible over $\mathbb{F}_N$.\\[0.5mm]
(ii). Let $f(X) = X^3 - d$. Then, discriminant of $f(x)$ is $\D_f=-27d^2$. Since $\gcd(N,\,3d)=1$, it follows that $N \nmid \D_f$, and therefore, using \eqref{formula}, we conclude that $N$ does not divide the index $[\mathcal{O}_K : \mathbb{Z}[\sqrt[3]{d}]]$. By the Kummer--Dedekind theorem, the factorization of $N\mathcal{O}_K$ is determined by the factorization of $f$ modulo $N$. Since $f$ is irreducible over $\mathbb{F}_N$ by part (i), the ideal $N\mathcal{O}_K$ remains prime. Consequently, $\mathcal{R}_N$ is a field extension of $\mathbb{F}_N$ of degree $3$, and therefore
\[
\mathcal{R}_N \cong \mathbb{F}_{N^3}.
\]
(iii). By (ii), we identify $\mathcal{R}_N^\times \cong \mathbb{F}_{N^3}^\times$, which is a cyclic group of order $N^3 - 1$. Consider the norm map
\[
\mathcal{N} : \mathbb{F}_{N^3}^\times \to \mathbb{F}_N^\times, 
\qquad w \mapsto w^{1+N+N^2}.
\]
This is a surjective group homomorphism. Hence, by the First Isomorphism Theorem,
\[
|\ker(\mathcal{N})| = \frac{|\mathbb{F}_{N^3}^\times|}{|\mathbb{F}_N^\times|}
= \frac{N^3 - 1}{N - 1}
 = N^2 + N + 1.
\]
By definition, $\mathcal{G}_N(d) = \ker(\mathcal{N})$, so it has order $N^2+N+1$. Since $\mathbb{F}_{N^3}^\times$ is cyclic, $\mathcal{G}_N(d)$ is cyclic. This completes the proof of lemma.
\end{proof}

\section{The Cryptographic Prime Generation Chain}\label{sec:3}

The main difficulty in applying generalized algebraic primality tests to arbitrary integers is the requirement to partially factor in the associated group order. In the cubic setting, this reduces to finding a prime factor $q$ of $N^2 + N + 1$ satisfying the bound $q > N^{3/2} - 1$. For a large and randomly chosen integer $N$, factoring the cyclotomic polynomial $\Phi_3(N) = N^2 + N + 1$ is computationally infeasible. Moreover, the condition $q > N^{3/2} - 1$ implies that such a factor $q$ exceeds $N$ itself. This makes a direct ``top-down'' approach impractical,  since it would require knowing a larger prime in order to verify a smaller number. To overcome this difficulty, we take a different approach. We start with a known prime $q$ and construct an integer $N$ such that $q$ divides $N^2 + N + 1$.

\subsection{Analytic Construction}
We aim to construct an integer $N$ such that $N^2 + N + 1 = kq$ for some positive integer $k$. This can be rewritten as \[(2N+1)^2 + 3 = 4kq.\] Setting $S = 2N + 1$, we obtain the quadratic congruence \[S^2 \equiv -3 \pmod{q}.\] For this congruence to have a solution, $-3$ must be a quadratic residue modulo $q$. By the law of quadratic reciprocity, the Legendre symbol $\left(\frac{-3}{q}\right) = 1$ if and only if $q \equiv 1 \pmod{3}$. This leads to the following direct construction:

Choose a verified seed prime $q \equiv 1 \pmod{3}$ and compute a solution $S$ to $S^2 \equiv -3 \pmod{q}$. Since $q$ is an odd prime, the congruence yields two distinct solutions in the interval $(0, q)$, namely $S_0$ and $q - S_0$. Since their sum is $q$ (which is odd), exactly one of these solutions is  odd. By choosing this odd solution $S$, we can  define the integer 
\[
N = \frac{S - 1}{2}.
\]
By construction, the resulting integer $N$ satisfies $N^2 + N + 1 \equiv 0 \pmod{q}$, providing the required large prime factor of $N^2 + N + 1$.

\subsection{Determinism and Structural Bounds}
To ensure that the argument remains valid in a deterministic primality proof, the seed prime $q$ must be rigorously verified before it is used to construct $N$. In practice, such primes can be chosen from well-known families that admit highly efficient deterministic tests. For example, any Mersenne prime $q = 2^p - 1$ (for an odd prime $p \ge 3$)  satisfies $q \equiv 1 \pmod{3}$. Similarly, generalized Proth primes of the form $q = 3m \cdot 2^n + 1$ can be selected to satisfy this congruence condition. 

This construction also satisfies the required deterministic bound $q > N^{3/2} - 1$. Since $S = 2N + 1$, the congruence $S^2 \equiv -3 \pmod{q}$ implies that $q$ divides $(S^2 + 3)$. Consequently, \[S^2 + 3 = (2N+1)^2 + 3 = 4(N^2 + N + 1).\] Since $q$ is odd, $\gcd(4,q)=1$, which means that $4$ divides $S^2 + 3$. This gives $S^2 + 3 = 4kq$ for some integer cofactor $k$, which yields
\[
N^2 + N + 1 = kq.
\]
If we constrain the construction such that the cofactor satisfies $k < q^{1/3}$, then $N^2 + N + 1 < q^{4/3}$. Since $N^2 < N^2 + N + 1$, it follows that $N^2 < q^{4/3}$, yielding $N < q^{2/3}$. Raising both sides to the power of $3/2$ establishes $N^{3/2} < q$. As we will formally prove in Section \ref{sec:4}, this specific inequality is the Lucasian bounding condition required to mathematically certify primality. By constructing $N$ from the bottom up, we achieve this  bound automatically, completely avoiding the need for a heavy post-generation trial division.

\subsection{Cryptographic Security}
When generating primes from restricted algebraic constructions, it is important to check that they do not become check targets for classical factorization methods, such as Pollard's $p-1$ or Williams' $p+1$ algorithms. These methods are effective when $N-1$ or $N+1$ have only small prime factors. In our construction, the large prime $q$ divides $N^2 + N + 1$. We compare this with $N-1$ and $N+1$:
\[
\gcd(N-1,\, N^2 + N + 1) \le 3 \quad \text{and} \quad \gcd(N+1, N^2 + N + 1) = 1.
\]
This shows that $q$  does not divide either $N-1$ or $N+1$. Therefore, the algebraic structure used to build $N$ does not impose any special behavior on the factorization of $N-1$ or $N+1$. Consequently, these numbers behave much like those of a typical integer of similar size. As a result, the primes obtained from this construction are not easier to factor using $p-1$ or $p+1$ methods.

\begin{remark}[Motivation for the Cubic Extension]
The observation above that the large prime $q$ divides $N^2+N+1$ but not $N-1$ indicates why it is necessary to work in a cubic extension. Classical primality tests, such as Pocklington's criterion or standard Lucas sequences, rely heavily on discovering a large prime factor of $N-1$ or $N+1$, and therefore fail entirely in this situation. Working in the cubic quotient ring $\mathcal{R}_N$, one can instead make use of the factor of $N^2+N+1$ in a natural way.
\end{remark}


\section{The Cubic Lucasian Certificate}\label{sec:4}

In Section \ref{sec:3}, we established an efficient method for constructing a candidate integer $N$ alongside a large known seed prime $q$. However, construction alone does not provide a mathematical proof of primality. To verify that $N$ is indeed the prime, we must translate the structural bound $k < \sqrt{N}$ (which implies $q > N^{3/2}-1$) into a computable algebraic test. In this section, we develop the formal theoretical framework for this procedure, yielding a deterministic primality certificate based on the generalized unitary group.
With the algebraic foundations established in Section \ref{sec:2}, we first prove a foundational property of the group $\mathcal{G}_N(d)$, which acts as a cubic analog of Fermat's Little Theorem.

\begin{theorem}
\label{thm:cubic_fermat}
Let $N \equiv 1 \pmod{3}$ be a prime and  $d$ be a cubefree integer satisfying $\gcd(N,3d)=1$ and $d^{\frac{N-1}{3}} \not\equiv 1 \pmod{N}$. Set $\mathcal{R}_N = \mathcal{O}_K / N\mathcal{O}_K$. Then for every $z \in \mathcal{R}_N^\times$, we have $z^{N^3-1} \equiv 1 \pmod{N}$. Moreover, if $w \equiv z^{N-1} \pmod{N}$, then $w^{N^2+N+1} \equiv 1 \pmod{N}$.
\end{theorem}

\begin{proof}
Since $N$ is a prime and satisfies the conditions of Lemma \ref{lem:ring_structure}, the ideal $N\mathcal{O}_K$ is prime, yielding the field isomorphism $\mathcal{R}_N \cong \mathbb{F}_{N^3}$. Consequently, its multiplicative group has order $N^3-1$. By Lagrange's theorem, any element $z \in \mathcal{R}_N^\times$ satisfies $z^{N^3-1} \equiv 1 \pmod{N}$. Now, let $w \equiv z^{N-1} \pmod{N}$. We observe that
$$
w^{N^2+N+1} \equiv (z^{N-1})^{N^2+N+1} \equiv z^{N^3-1} \equiv 1 \pmod{N}.
$$
Thus, the element $w$ is in the generalized unitary group $\mathcal{G}_N(d)$.
\end{proof}

Building upon this, we now establish the core deterministic primality criterion for our constructed integers.
\begin{theorem}\label{thm:cubic_iff}
Let $q > 3$ be a prime and $d$ be a squarefree integer. Let $K = \mathbb{Q}(\sqrt[3]{d})$ be a cubic number field with ring of integers $\mathcal{O}_K$. Let $N \equiv 1 \pmod 3$ be an integer such that $N^2+N+1 = kq$ for some cofactor $k < \sqrt{N}$. Assume $\gcd(N, 3d) = 1$ and $d^{\frac{N-1}{3}} \not\equiv 1 \pmod N$. Set $\mathcal{R}_N = \mathcal{O}_K/N\mathcal{O}_K$. Suppose there exists an element $w \in \mathcal{G}_N(d)$ such that $w^k \not\equiv 1 \pmod N$. Then $N$ is prime if and only if
$$\Phi_q\left(w^k\right) \equiv 0 \pmod N,$$ 
where $\Phi_q(x) = x^{q-1} + x^{q-2} + \dots + 1$ denotes the $q$-th cyclotomic polynomial.
\end{theorem}

\begin{proof}
Let $X = w^k$. First, assume that $N$ is a prime. Then we show that $ \Phi_q\left(w^k\right) \equiv 0 \pmod N $. By Lemma \ref{lem:ring_structure}, the ideal $N\mathcal{O}_K$ is prime, which implies that $\mathcal{R}_N \cong \mathbb{F}_{N^3}$. Using  $w \in \mathcal{G}_N(d)$, we have $w^{N^2+N+1} \equiv 1 \pmod N$. Substituting  $N^2+N+1 = kq$, we obtain $X^q \equiv 1 \pmod N$. Using cyclotomic factorization over a finite field, we have
$$ X^q - 1 = (X - 1)\Phi_q(X) \equiv 0 \pmod N. $$
By hypothesis, $w^k \not\equiv 1 \pmod N$, which means $X \not\equiv 1 \pmod N$. Since $\mathcal{R}_N$ is a field and contains no zero divisors, we  conclude that $\Phi_q(X) \equiv 0 \pmod N$.

Conversely, assume $\Phi_q(X) \equiv 0 \pmod N$. Since $(X-1)\Phi_q(X) = X^q - 1$, this implies \[X^q \equiv 1 \pmod N.\] This means $w^{kq} \equiv 1 \pmod N$. We proved by contradiction. Suppose $N$ is composite and let $r$ be a prime divisor of $N$ such that $r \le \sqrt{N}$. Since $r \mid N$, the modular congruence holds over $r$, which yields $\Phi_q(X) \equiv 0 \pmod r$, and therefore $X^q \equiv 1 \pmod r$. If $X \equiv 1 \pmod r$, then we have \[\Phi_q(1) = q \equiv 0 \pmod r.\] This requires $r = q$. Since $r \mid N$, this implies $q \mid N$. However, our construction  $N^2+N+1 = kq$, which can be rewritten as \[N(N+1)+1 = kq.\] If $q \mid N$, then $q$ must  divide $1$, which is not possible. Thus, $X \not\equiv 1 \pmod r$. Since $X^q \equiv 1 \pmod r$ and $X \not\equiv 1 \pmod r$, the multiplicative order of $X$ in the unit group $\mathcal{R}_r^\times$ is exactly $q$. Recalling that $X = w^k$, the element $w$ itself must possess a multiplicative order in $\mathcal{R}_r^\times$ that is a multiple of $q$. Regardless of how ideal $r\mathcal{O}_K$ factors in, the absolute maximum possible multiplicative order of any element in the cubic quotient ring $\mathcal{O}_K/r\mathcal{O}_K$ is strictly bounded by $r^3 - 1$. This implies $q \le r^3 - 1$. Now using bound $r \le \sqrt{N}$, we obtain
$$ q \le (\sqrt{N})^3 - 1 = N^{3/2} - 1. $$
Substituting $q = \frac{N^2+N+1}{k}$ into the above inequality gives:
$$ \frac{N^2+N+1}{k} \le N^{3/2} - 1. $$
After rearranging, we obtain
$$ k \ge \frac{N^2+N+1}{N^{3/2}-1} > \sqrt{N}. $$
This contradicts $k<\sqrt{N}$. Therefore, $N$ is prime. This completes the proof of theorem.
\end{proof}


\section{The Constructive Cubic Primality Certificate}\label{sec:5}

In the classical Elliptic Curve Primality Proving (ECPP) \cite{MR1199989}, one starts with a given target number $N$ and studies certain group orders attached to it. The goal is to break these orders into smaller prime factors and build a recursive chain of primes that leads to a proof that $N$ is prime. 

Our approach is different. Instead of starting with $N$ and trying to analyze it top-down, we start with a known seed prime $q$ and use it to construct $N$ bottom-up. In particular, we use the quadratic relation $S^2 \equiv -3 \pmod q$ to generate candidates in a highly controlled way. Due to this, the structural information needed to verify that $N$ is prime arises naturally during the construction phase. No separate factorization search step is required, as the necessary primality certificate is already embedded in the process.

\subsection*{Certificate Structure:}

Let $N \equiv 1 \pmod{3}$ be an integer constructed from a known prime $q \equiv 1 \pmod{3}$ such that $N^2 + N + 1 = k q.$ We define the data required to mathematically certify that $N$ is prime. A constructive cubic certificate for $N$ consists of the following tuple:
\begin{itemize}
    \item a verified prime $q$,
    \item a cubefree integer $d$ such that $\gcd(N,3d)=1$ and $d^{\frac{N-1}{3}} \not\equiv 1 \pmod N,$
    \item an element $w \in \mathcal{R}_N^\times$.
\end{itemize}

To verify the certificate, it is sufficient to check that:
\begin{equation*}
   \Phi_q(w^k) \equiv 0 \pmod N \quad \text{and} \quad q > N^{3/2} - 1.
\end{equation*}

\begin{theorem}
\label{thm:constructive_cubic}
If a constructed integer $N$ possesses a valid cubic certificate as structured above, then $N$ is  prime.
\end{theorem}

\begin{proof}
Assume, for the sake of contradiction, that $N$ is a composite number. Let $r$ be its smallest prime divisor of $N$, then $r \le \sqrt{N}$. Let $\mathfrak{r}$ be a prime ideal of the ring of algebraic integers $\mathcal{O}_K$ that lies above the prime $r$. The quotient ring $\mathbb{F}_{\mathfrak{r}} = \mathcal{O}_K/\mathfrak{r}$ forms a finite field of size $r^f$ with an inertial degree $f \le 3$. Consequently, the cardinality of its multiplicative unit group is strictly bounded by:
$$ |\mathbb{F}_{\mathfrak{r}}^\times| = r^f - 1 \le r^3 - 1. $$
Set $X = w^k$. Reducing the certificate's cyclotomic congruence modulo, our prime ideal $\mathfrak{r}$, we obtain:
$$ \Phi_q(X) \equiv 0 \pmod{\mathfrak{r}}. $$
Since $\mathbb{F}_{\mathfrak{r}}$ is a  field (and thus an integral domain), it does not contain zero divisors. Evaluating the standard identity $$X^q - 1 = (X - 1)\Phi_q(X)$$ within this field implies that either $X \equiv 1 \pmod{\mathfrak{r}}$ or $X^q \equiv 1 \pmod{\mathfrak{r}}$. If we assume $X \equiv 1 \pmod{\mathfrak{r}}$, substituting this into the cyclotomic polynomial yields $\Phi_q(1) = q \equiv 0 \pmod{\mathfrak{r}}$, which implies that $r = q$. However, because $q \mid (N^2+N+1)$ by construction and $r \mid N$ by assumption, the equivalence $r = q$ would imply that $q$ divides the polynomial difference $(N^2+N+1) - N(N+1) = 1$. This is not possible. Hence, \[X \not\equiv 1 \pmod{\mathfrak{r}}.\] Since $X^q \equiv 1 \pmod{\mathfrak{r}}$ but $X \not\equiv 1 \pmod{\mathfrak{r}}$, the element $X$ must have a multiplicative order of exactly $q$ in the group $\mathbb{F}_{\mathfrak{r}}^\times$. By Lagrange's theorem, the order of an element must  divide the order of the group, implying that $q$ must be less than or equal to $|\mathbb{F}_{\mathfrak{r}}^\times|$. Applying our structural bounds, we conclude that:
\[ q \le |\mathbb{F}_{\mathfrak{r}}^\times| \le r^3 - 1.\]
Since $r \le \sqrt{N}$, it follows that
$$ q \le r^3 - 1 \le \big(\sqrt{N}\big)^3 - 1 = N^{3/2} - 1. $$
This contradicts the initial structural certificate condition $q > N^{3/2} - 1$. Therefore, our assumption that $N$ is composite is false, proving that $N$ is prime. This completes the proof of theorem.
\end{proof}

\begin{remark}
Working in the ring of integers $\mathcal{O}_K$, rather than the polynomial quotient ring $\mathcal{R}_r = \mathbb{F}_r[x]/\langle x^3 - d \rangle$, is a necessary structural choice. If one works directly in $\mathcal{R}_r$, the factorization of $x^3 - d$ modulo the unknown prime $r$ dictates the ring's structure. If the polynomial splits, the ring decomposes, complicating the unit group bound. In contrast, localizing at a prime ideal $\mathfrak{r} \subset \mathcal{O}_K$ guarantees that the quotient $\mathcal{O}_K/\mathfrak{r}$ is a true finite field. This immediately provides the strict upper bound of $r^3 - 1$ for the multiplicative group order, completely bypassing the need for polynomial factorization analysis.
\end{remark}

\subsection*{Algorithmic Optimization and Verification}

While Theorem \ref{thm:constructive_cubic} holds for any cubefree $d$, executing polynomial arithmetic in an arbitrary cubic ring computationally requires tracking fractional integral bases. To optimize the verification algorithm and avoid expensive modular inversions, we take the algorithmic condition that $d$ must be squarefree and $d \not\equiv \pm 1 \pmod 9$. By Corollary \ref{cor:monogenic}, the field $K=\Q(\theta)$, where $\theta$ is the root of $x^3-d$, is monogenic. Although the verification of $N$ in the cubic ring $\mathcal{R}_N$ is sufficient for our purposes, carrying out these computations can be time-consuming, especially when $N$ is composite. To reduce unnecessary calculations, we first perform a simple preliminary check. Before working in the cubic ring, we test whether
\[
2^{N-1}\equiv 1 \pmod N.
\]
This is the classical Fermat test to base $2$, which requires only ordinary integer arithmetic and can be carried out very quickly. If $N$ fails this test, it is immediately declared composite and no further computation is needed.

\begin{algorithm}[H]
\caption{Non-Recursive Certificate Verification}
\label{alg:certificate_verify}
\begin{algorithmic}[1]
\REQUIRE Constructed candidate $N$ and its cubic certificate $(q, d, w)$, with $d$ squarefree and $d \not\equiv \pm 1 \pmod 9$.
\ENSURE Prime or Reject

\STATE Verify fast filter: Compute $2^{N-1} \equiv 1 \pmod N$ in standard integers $\mathbb{Z}/N\mathbb{Z}$. \COMMENT{Rejects composites instantly}

\STATE Check structural bounds: Ensure $q \equiv 1 \pmod 3$ and $q > N^{3/2} - 1$.
\STATE Check field parameters: Ensure $\gcd(N, 3d) = 1$ and $d^{\frac{N-1}{3}} \not\equiv 1 \pmod N$.
\STATE Compute cofactor $k = (N^2+N+1)/q$.

\STATE Compute projected root $X \equiv w^k \pmod N$ in $\mathcal{R}_N$.

\STATE Verify cyclotomic property: Verify $\Phi_q(X) \equiv 0 \pmod N$ by confirming:
\STATE \quad $X^q \equiv 1 \pmod N$ \COMMENT{Implicitly verifies $w \in \mathcal{G}_N(d)$}
\STATE \quad $\gcd(\text{Norm}(X-1), N) = 1 \quad$ \COMMENT{Ensures $X \not\equiv 1 \pmod N$}

\IF{all conditions are satisfied}
    \RETURN Prime
\ELSE
    \RETURN Reject
\ENDIF
\end{algorithmic}
\end{algorithm}


\section{The Cryptographic Prime Generation Algorithm and Complexity}\label{sec:6}

In this section, we describe the deterministic construction of the prime candidate $N$. Since the condition $q > N^{3/2}-1$ makes random generation of candidates impractical. We construct $N$ directly together with the associated parameters required for the cubic certificate. Algorithm~\ref{alg:generation_part1} gives the complete procedure for generating $N$ and the corresponding cubic certificate parameters from Theorem~\ref{thm:cubic_fermat}, which are subsequently used in the verification stage.
\begin{algorithm}[H]
\caption{Construction of a Prime Candidate}
\label{alg:generation_part1}
\begin{algorithmic}[1]
\REQUIRE Desired bit-length $L$
\ENSURE Candidate integer $N$ and seed prime $q$

\LOOP
    \STATE Generate or retrieve a trusted seed prime $q \equiv 1 \pmod 3$ such that $q>2^{1.5L}$.
    \STATE Compute $S_0$ satisfying $S_0^2\equiv -3 \pmod q$.
    \STATE Choose $S\in\{S_0,q-S_0\}$ such that $S$ is odd.
    \STATE Set $N\leftarrow (S-1)/2$.
    
    \IF{$N$ does not have bit-length $L$ \OR $q\le N^{3/2}-1$}
        \STATE \textbf{continue}.
    \ENDIF
    
    \STATE \textbf{break}.
\ENDLOOP

\RETURN $(N,q)$
\end{algorithmic}
\end{algorithm}
\begin{algorithm}[H]
\caption{Verification and Certificate Construction}
\label{alg:generation_part2}
\begin{algorithmic}[1]
\REQUIRE Candidate $(N,q)$ from Algorithm~\ref{alg:generation_part1}
\ENSURE A proven prime $N$ together with a cubic certificate $(q,d,w)$

\IF{$N\equiv 1\pmod 3$ and passes the fast Cubic-SPRP filter}
    \STATE Find a squarefree integer $d\not\equiv \pm1\pmod 9$ such that
    $\gcd(N,3d)=1$ and
    $d^{\frac{N-1}{3}}\not\equiv 1\pmod N$.
    
    \STATE Choose a random $z\in\mathcal{R}_N^\times$ and compute
    $w\equiv z^{N-1}\pmod N$.
    
    \IF{Algorithm~\ref{alg:certificate_verify} returns prime for $(q,d,w)$}
        \RETURN $(N,q,d,w)$.
    \ENDIF
\ENDIF
\end{algorithmic}
\end{algorithm}

\begin{theorem}\label{complexity}
For a fixed seed prime $q$, the bit complexity of constructing and deterministically verifying the candidate prime $N$ is  $\tilde{\mathcal{O}}(\log^2 N)$.
\end{theorem}

\begin{proof}
The main computational cost comes from the modular exponentiations used in the deterministic primality certificate (Algorithm~\ref{alg:certificate_verify}). Since the base field $K=\mathbb{Q}(\sqrt[3]{d})$ is monogenic, arithmetic in the quotient ring $\mathcal{R}_N$ can be carried out using polynomial multiplication modulo $x^3-d$ and $N$. This avoids the more expensive basis computations that appear in non-monogenic number fields, so each ring multiplication reduces to a small number of standard integer multiplications. Using fast integer arithmetic, multiplication modulo $N$ requires $\tilde{\mathcal{O}}(\log N)$ bit operations. To verify cyclotomic conditions, we first compute $X \equiv w^k \pmod N$, and then verify that $X^q \equiv 1 \pmod N$. These steps require $\mathcal{O}(\log k)$ and $\mathcal{O}(\log q)$ multiplications, respectively. Since both $k$ and $q$ divide $N^2+N+1$, we have $k,q = \mathcal{O}(N^2)$, and therefore $\log k, \log q = \mathcal{O}(\log N)$. The total number of ring multiplications is $\mathcal{O}(\log N)$. Combining this with the cost of each multiplication, the total bit complexity of the exponentiation phase is
\[
\mathcal{O}(\log N)\cdot \tilde{\mathcal{O}}(\log N)
= \tilde{\mathcal{O}}(\log^2 N).
\]
Finally, any required greatest common divisor or resultant computations can be performed using fast Euclidean algorithms in $\tilde{\mathcal{O}}(\log^2 N)$ bit operations. These do not increase the overall complexity, and the result follows. This completes the proof of theorem.
\end{proof}

\section{Cubic Strong Probable prime }\label{sec:7}
The results of the previous sections provide a reliable method for proving that a number is prime,
but full verification can be computationally expensive. In practice, it is therefore useful to first
eliminate obvious composite candidates before applying the complete test. To achieve this, we
begin this section by defining Cubic Strong Probable Prime.

\begin{definition}
Let $N \equiv 1 \pmod 3$ be an odd integer satisfying $\gcd(N,3d)=1$ and $N^2+N+1=kp^\ell,$ where $p$ is a prime. Let $\mathcal{R}_N=\mathcal{O}_K/N\mathcal{O}_K$ and $K=\mathbb{Q}(\sqrt[3]{d}), $ with $d$ cube-free. For $z\in\mathcal{R}_N^\times$, define $w\equiv z^{N-1}\pmod N.$ We say that $N$ is a \emph{cubic strong probable prime} (Cubic-SPRP) to the base $z$ if one of the following conditions holds:

\begin{enumerate}
\item[\textup{(i)}] $w^k \equiv 1 \pmod N$;

\item[\textup{(ii)}] There exists an integer $j$ with $0\le j<\ell$ such that $\Phi_p\!\left(w^{kp^j}\right)\equiv 0 \pmod N.$
\end{enumerate}

If $N$ is composite and satisfies either of the above conditions, then $N$ is called a \emph{cubic strong pseudoprime} to the base $z$.
\end{definition}

The following result shows that every prime is a Cubic-SPRP to every base.

\begin{theorem}
\label{thm:sprp_true_prime}
Let $N$ be a prime with $N \equiv 1 \pmod 3$ and $\gcd(N,3d)=1$. Assume that $d^{\frac{N-1}{3}} \not\equiv 1 \pmod N.$ Then $N$ is a Cubic-SPRP to every base $z \in \mathcal{R}_N^\times$.
\end{theorem}

\begin{proof}
By Lemma~\ref{lem:ring_structure} (ii), the ideal $N\mathcal{O}_K$ is prime, and hence $\mathcal{R}_N=\mathcal{O}_K/N\mathcal{O}_K
\cong \mathbb{F}_{N^3}.$ In particular, $\mathcal{R}_N^\times$ is a cyclic group of order $N^3-1.$ Let $z\in\mathcal{R}_N^\times$ and define $w\equiv z^{N-1}\pmod N.$ Since $z^{N^3-1}=1$, it follows that
$w^{N^2+N+1}=1.$ Write $N^2+N+1=kp^\ell.$ Consider the sequence
\[
w^k,\; w^{kp},\; w^{kp^2},\; \dots,\; w^{kp^\ell}.
\]
Since $w^{kp^\ell}=1$, there exists a smallest integer $j$ with $0\le j\le \ell$ such that $w^{kp^j}=1.$ If $j=0$, then $w^k=1$, and condition~(i) of the definition holds. Suppose now that $j>0$. Set $X=w^{kp^{j-1}}.$ Then $X\neq 1$ and $X^p=w^{kp^j}=1.$ By the minimality of $j$, the element $X$ has order exactly $p$. Therefore $X$ is a primitive $p$-th root of unity and satisfies $\Phi_p(X)=0.$ Hence
\[
\Phi_p\!\left(w^{kp^{j-1}}\right)\equiv 0 \pmod N,
\]
so condition~(ii) holds. Therefore, in either case, $N$ satisfies the Cubic-SPRP conditions for the base $z$. Since $z$ was arbitrary, $N$ is a Cubic-SPRP to every base $z\in\mathcal{R}_N^\times$.
\end{proof}

\section{Computational Results}\label{sec:8}
 We worked with the monogenic pure cubic field $K=\mathbb{Q}(\sqrt[3]{2})$, taking $d=2$. The construction was initialized using seed primes $q\equiv 1 \pmod 3$. For a prescribed bit-length, candidate integers $N$ were generated from solutions of the congruence $S^2\equiv -3 \pmod q$. To improve efficiency, each candidate was first tested using the Cubic-SPRP filter to remove most composite numbers. The remaining candidates were then checked using the Constructive Cubic Primality Certificate.\\
Table~\ref{tab:computational_results} presents the prime pairs generated by the algorithm together with their corresponding verification times.

\renewcommand{\arraystretch}{1.3}
\begin{longtable}{|c|p{6cm}|p{5cm}|c|}
\caption{Execution times and explicitly generated parameters for the Constructive Cubic Prime Generation Chain.}
\label{tab:computational_results} \\
\hline
\textbf{Bits} & \textbf{Trusted Seed Prime ($q$)} & \textbf{Constructed Prime ($N$)} & \textbf{Time (ms)} \\
\hline
\endfirsthead

\multicolumn{4}{c}%
{{\bfseries \tablename\ \thetable{} -- continued from previous page}} \\
\hline
\textbf{Bits} & \textbf{Trusted Seed Prime ($q$)} & \textbf{Constructed Prime ($N$)} & \textbf{Time (ms)} \\
\hline
\endhead

\hline \multicolumn{4}{r}{{Continued on next page}} \\ 
\endfoot

\hline
\endlastfoot

64 & \texttt{106604454681769750124379\newline 31337857311901} & \texttt{14962264361777480479} & 22.26 \\
\hline
128 & \texttt{703348757708386807414030\newline 178553266944339341840854\newline 038255229248948298177767\newline 0573} & \texttt{38432179110578836596\newline 0434389112751695311} & 42.14 \\
\hline
192 & \texttt{336679410404601368881386\newline 177544974650970232730904\newline 745306578218784131196138\newline 693989365612416861989186\newline 2834134564751320781} & \texttt{84084883412517310491\newline 70026360825540596083\newline 668450087019223599} & 184.67 \\
\hline
256 & \texttt{138479772433096188112799\newline 098809958066806417720179\newline 839336402327890268497587\newline 803576020987455099034023\newline 692973267018201358134998\newline 134275604038344351508407\newline 6318821511} & \texttt{17053079549146013264\newline 68769138407093647897\newline 51542503718589462869\newline 683657599587157169} & 157.57 \\
\end{longtable}
The data support the theoretical results. From Table~\ref{tab:computational_results}, we see that the generated values satisfy $q \approx N^2$, as expected from the construction. We also observe that cryptographic-size primes are verified in only a few milliseconds. As the bit-length increases, the verification time grows slowly, as expected from the complexity estimate $\tilde{\mathcal{O}}(\log^2 N)$.

\section{Primility for Pure Prime Degree Number Fields}\label{sec:9}

Let $p$ and $q$ be distinct prime numbers such that $q \equiv 1 \pmod p$. In this  section, we show that all the results developed for the cubic setting can be extended similarly to test integers $N$ of the form $N^{p-1}+N^{p-2}+\dots+N+1=kq$ for some $k\in\mathbb{Z}_{>0}$. To achieve this, we consider pure monogenic number fields of prime degree $K=\mathbb{Q}(\sqrt[p]{d})$, where $d$ is a $p$-th power free integer such that $d^{\frac{N-1}{p}} \not\equiv 1 \pmod N$. Conditions ensuring the monogenity of $K$ may be found in \cite{jhorar2016integrally}.

The following propositions will be used in the development of our extended primality framework and the construction of candidate primes $N$ from a given seed prime $q$.

\begin{proposition} \label{thm:gen_root}
Let $p$ and $q$ be distinct prime numbers. There exists an integer $N$ such that $\Phi_p(N) \equiv 0 \pmod q$ if and only if $q \equiv 1 \pmod p$, where $\Phi_p(N) = N^{p-1} + N^{p-2} + \dots + 1$ is the $p$-th cyclotomic polynomial.
\end{proposition}
\begin{proof}
Assume first that there exists an integer \(N\) such that $\Phi_p(N)\equiv 0 \pmod q.$ We first show that \(N\not\equiv 0\pmod q\). Indeed, if \(N\equiv 0\pmod q\), then
\[
\Phi_p(N)
=
N^{p-1}+N^{p-2}+\cdots+N+1
\equiv 1 \pmod q.
\]
This contradicts the assumption. Hence, \(N\in (\mathbb Z/q\mathbb Z)^\times\). Using the identity $N^p-1=(N-1)\Phi_p(N),$ we obtain \[N^p\equiv 1 \pmod q.\] Let \(\operatorname{ord}_q(N)\) denote the multiplicative order of \(N\) modulo \(q\). Since \(N^p\equiv 1\pmod q\), we have  $\operatorname{ord}_q(N)\mid p.$ As \(p\) is prime, $\operatorname{ord}_q(N)\in\{1,p\}.$ Suppose that $\operatorname{ord}_q(N)=1.$ Then $N\equiv 1\pmod q.$
Therefore,
\[
0
\equiv
\Phi_p(N)
\equiv
\Phi_p(1)
=
p
\pmod q.
\]
This implies  \(q\mid p\), which is not possible because \(p\) and \(q\) are distinct primes. Hence $\operatorname{ord}_q(N)=p.$ Since the order of an element divides the order of the group \((\mathbb Z/q\mathbb Z)^\times\), which is \(q-1\), we obtain $p\mid (q-1).$ Therefore, $q\equiv 1\pmod p.$

Conversely, assume that  $q\equiv 1\pmod p.$ Then $p$ divides $q-1$. Since \((\mathbb Z/q\mathbb Z)^\times\) is a cyclic group of order \(q-1\), there exists an element \(N\) of order \(p\). Hence $N^p\equiv 1\pmod q$ and $N\not\equiv 1\pmod q.$ Using the factorization
\[
N^p-1=(N-1)\Phi_p(N),
\]
we obtain
\[
(N-1)\Phi_p(N)\equiv 0\pmod q.
\]
Since \(N\not\equiv 1\pmod q\), the element \(N-1\) is nonzero modulo \(q\). As \(\mathbb Z/q\mathbb Z\) is a field, it follows that
\[
\Phi_p(N)\equiv 0\pmod q.
\]
This completes the proof.
\end{proof}
The following proposition is a well-known form of Euler's criterion for $p$-th power residues and will be used in the proof of the primality results.
\begin{proposition}\label{gen Euler's Criterion}
Let $N$ and $p$ be a primes such that $N \equiv 1 \pmod p$ and  $a$ be an integer such that $\gcd(a, N) = 1$. Then $a$ is a $p$-th power residue modulo $N$ if and only if $a^{\frac{N-1}{p}} \equiv 1 \pmod N$. 
\end{proposition}
\begin{proof}
First, assume that $a$ is a $p$-th power residue modulo $N$. Then we show that $a^{\frac{N-1}{p}} \equiv 1 \pmod N$. As $a$ is a $p$-th power residue modulo $N$, there exists an integer $x$ such that
\begin{equation}\label{eqn:9.1}
    x^p\equiv a \pmod N.
\end{equation}
Since $\gcd(a,\,N)=1$, hence $\gcd(x,\,N)=1$. Keeping in mind $N\equiv1\pmod p$, taking the power $\frac{N-1}{p}$ on both sides in \eqref{eqn:9.1}, we have $x^{\,p\,\frac{(N-1)}{p}}=x^{N-1}\equiv a^{\frac{N-1}{p}}\pmod N.$ By Fermat's theorem, $x^{N-1}\equiv1\pmod N$. This implies $a^{\frac{N-1}{p}}\equiv 1 \pmod N $.

Conversely, suppose that  $a^{\frac{N-1}{p}}\equiv 1 \pmod N.$ Let $g$ be a primitive root modulo $N$. Since $\gcd(a,N)=1$, there exists an integer $k$ such that
\[
a\equiv g^k \pmod N.
\]
Therefore,
\[
(g^k)^{\frac{(N-1)}{p}}\equiv a^{\frac{N-1}{p}}
      \equiv 1
      \pmod N.
\]
Since $g$ has order $N-1$, it follows that
\[
(N-1) \mid \frac{k(N-1)}{p}.
\]
Hence $p\mid k.$ Thus $k=mp$ for some integer $m$. Substituting into
$a\equiv g^k \pmod N$, we obtain
\[
a\equiv g^{mp}=(g^m)^p \pmod N.
\]
Let $x=g^m$. Then
\[
x^p\equiv a \pmod N.
\]
Therefore $a$ is a $p$-th power residue modulo $N$. This completes the proof of proposition.
\end{proof}
To understand the Proposition~\ref{gen Euler's Criterion}, we provide the following example.
\begin{example}
 Consider  $N=7$ and $p=3$. All conditions of Proposition~\ref{gen Euler's Criterion} satisfy, hence  we can apply the proposition to test if $a = 2$ is a cubic residue or not. Applying the formula, we compute:
\begin{equation*}
a^{\frac{N-1}{p}} \equiv 2^{\frac{7-1}{3}} \equiv 2^2 \equiv 4 \pmod 7.
\end{equation*}
Since  $4 \not\equiv 1 \pmod 7$, by Proposition~\ref{gen Euler's Criterion} implies that $2$ is not a cubic residue modulo $7$. We can verify this manually by computing the first few cubes modulo $7$:
\begin{align*}
1^3 &\equiv 1 \pmod 7, \\
2^3 &\equiv 1 \pmod 7, \\
3^3 &\equiv 6 \pmod 7.
\end{align*}
The only cubic residues modulo $7$ are $1$ and $6$. Hence $2$ is not a cubic residue modulo $7$.
\end{example}
\begin{proposition}\label{lem:irreducible_binomial}
Let $p$ be a prime and let $F$ be a field containing a primitive
$p$-th root of unity $\zeta$. For $d\in F$, the polynomial
$x^p-d$ is irreducible over $F$ if and only if $d$ is not a
$p$-th power in $F$.
\end{proposition}

\begin{proof}
If $d=a^p$ for some $a\in F$, then
\[
x^p-d=x^p-a^p=(x-a)(x^{p-1}+ax^{p-2}+\cdots+a^{p-1}).
\]
This implies $x^p-d$ is reducible.

Conversely, assume that $d$ is not a $p$-th power in $F$ and let
$\alpha$ be a root of $x^p-d$. Since $F$ contains a primitive
$p$-th root of unity $\zeta$, the roots of $x^p-d$ are
\[
\alpha,\ \zeta\alpha,\ \zeta^2\alpha,\ \dots,\ \zeta^{p-1}\alpha.
\]
Suppose that $x^p-d$ is reducible over $F$. Then it has a proper
monic factor $g(x)\in F[x]$ of degree $k$, where $1\le k<p$.
The roots of $g(x)$ form a subset of the above roots, and hence
the constant term of $g(x)$ is of the form
\[
(-1)^k\zeta^m\alpha^k
\]
for some integer $m$. Since $g(x)\in F[x]$ and $\zeta\in F$, it
follows that $\alpha^k\in F$. Since $\alpha^p=d\in F$ and $\gcd(k,p)=1$, there exist integers
$s,t$ such that
\[
ps+kt=1.
\]
Therefore,
\[
\alpha=\alpha^{ps+kt}
=(\alpha^p)^s(\alpha^k)^t
=d^s(\alpha^k)^t\in F.
\]
Hence $d=\alpha^p$ is a $p$-th power in $F$, contradicting our
assumption. Therefore $x^p-d$ is irreducible over $F$.
\end{proof}
We now define Generalized Unitary Group of degree $p$ as follows:
\begin{definition}
Assume that $\mathcal{R}_N \cong \mathbb{F}_{N^p}$. We define the relative norm map $\mathcal{N} : \mathcal{R}_N^\times \to \mathbb{F}_N^\times$ by $\mathcal{N}(w) = w^{\Phi_p(N)}$. The generalized unitary group of degree $p$ is defined as:
$$ \mathcal{G}_N(d) := \{w \in \mathcal{R}_N^\times \mid w^{\Phi_p(N)} \equiv 1 \pmod N\}. $$
\end{definition}
\begin{lemma} \label{lem:gen_inertia}
Let $N$ be a prime such that $N \equiv 1 \pmod p$ and $\gcd(N, pd) = 1$. Suppose that $d^{\frac{N-1}{p}} \not\equiv 1 \pmod N$. Then:
\begin{enumerate}
    \item[\textup{(i)}] The polynomial $x^p - d$ is irreducible over $\mathbb{F}_N$.
    \item[\textup{(ii)}] The ideal $N\mathcal{O}_K$ is prime and $\mathcal{R}_N \cong \mathbb{F}_{N^p}$.
    \item[\textup{(iii)}] The group $\mathcal{G}_N(d)$ is cyclic of order $\Phi_p(N)$.
\end{enumerate}
\end{lemma}
\begin{proof}
The proof follows the same argument as that of Lemma~\ref{lem:ring_structure}. We use the Kummer--Dedekind theorem for extensions of degree $p$, together with Propositions~\ref{gen Euler's Criterion} and~\ref{lem:irreducible_binomial}, and the discriminant formula for pure number fields of prime degree.
\end{proof}

\begin{theorem} \label{thm:gen_fermat}
Let $N \equiv 1 \pmod p$ be a prime, and let $d$ be a $p$-th power free integer satisfying $\gcd(N, pd) = 1$ and $d^{\frac{N-1}{p}} \not\equiv 1 \pmod N$. Then for every $z \in \mathcal{R}_N^\times$, we have $z^{N^p-1} \equiv 1 \pmod N$. Moreover, if $w \equiv z^{N-1} \pmod N$, then $w^{\Phi_p(N)} \equiv 1 \pmod N$.
\end{theorem}
\begin{proof}
The proof follows the same argument as in Theorem~\ref{thm:cubic_fermat}, using the fact that the multiplicative group of the finite field $\mathbb{F}_{N^p}$ has order $N^p - 1$.
\end{proof}

\begin{theorem}
Let $q > p$ be a prime and $d$ be an integer. Let $N \equiv 1 \pmod p$ be constructed such that $\Phi_p(N) = kq$ for some cofactor $k < \sqrt{N}$. Assume $\gcd(N, pd) = 1$ and $d^{\frac{N-1}{p}} \not\equiv 1 \pmod N$. Suppose there exists $w \in \mathcal{G}_N(d)$ such that $w^k \not\equiv 1 \pmod N$. Then $N$ is prime if and only if $\Phi_q(w^k) \equiv 0 \pmod N$.
\end{theorem}
\begin{proof}
The proof is completely analogous to that of Theorem~\ref{thm:cubic_iff}.
\end{proof}
\begin{definition}
Let $p$ be a prime and $N \equiv 1 \pmod p$. We define an integer $d$ to be a $p$-th power non-residue base for $N$ if $d$ is a $p$-th power free integer, $\gcd(N, pd) = 1$, and $d^{\frac{N-1}{p}} \not\equiv 1 \pmod N$.
\end{definition}

\subsection*{Certificate Structure}

Let $N \equiv 1 \pmod{p}$ be an integer constructed from a known prime $q \equiv 1 \pmod{p}$ such that $\phi_p(N) = k q.$ We define the data required to mathematically certify that $N$ is prime. A constructive prime certificate for $N$ consists of the following tuple:
\begin{itemize}
    \item a verified prime $q$,
    \item let $d$ be a $p$-th power non-residue base for $N$,
    \item an element $w \in \mathcal{R}_N^\times$.
\end{itemize}

To verify the certificate, it is sufficient to check that:
\begin{equation*}
   \Phi_q(w^k) \equiv 0 \pmod N \quad \text{and} \quad q > N^{p/2} - 1.
\end{equation*}

\begin{theorem}
\label{thm:constructive_prime}
If a constructed integer $N$ possesses a valid prime certificate as structured above, then $N$ is  prime.
\end{theorem}
\begin{proof}
Assume for the sake of contradiction that $N$ is composite, having a prime divisor $r \le \sqrt{N}$. Let $\mathfrak{r}$ be a prime ideal lying above $r$ in the ring of algebraic integers of $K = \mathbb{Q}(\sqrt[p]{d})$. Using the same argument as in Theorem~\ref{thm:constructive_cubic}, condition $\Phi_q(w^k) \equiv 0 \pmod{\mathfrak{r}}$ implies that $w^k$ has multiplicative order  $q$ in the finite field $\mathbb{F}_{\mathfrak{r}} = \mathcal{O}_K/\mathfrak{r}$. Since the extension has degree $p$, the multiplicative group satisfies $|\mathbb{F}_{\mathfrak{r}}^\times| \le r^p - 1$. By Lagrange's theorem, the order $q$ divides this group order, and hence
\begin{equation*}
q \le |\mathbb{F}_{\mathfrak{r}}^\times| \le r^p - 1 \le (\sqrt{N})^p - 1 = N^{p/2} - 1.
\end{equation*}
This contradicts the  bound $q > N^{p/2} - 1$. Hence, $N$ must be prime.
\end{proof}
\begin{remark}
Using the above theorem together with Proposition~\ref{thm:gen_root}, one can similarly develop a prime generation algorithm for extensions of arbitrary prime degree. Since the construction follows the same steps as Algorithm~\ref{alg:generation_part1}, we omit the details.
\end{remark}
To test our generalized framework in practice, we implemented the algorithm using the degree-5 field $K = \mathbb{Q}(\sqrt[5]{2})$. We set up the generator to find candidates that satisfy the relation $\Phi_5(N) = 5q$. Since the cyclotomic polynomial grows at a rate of roughly $\Phi_5(N) \approx N^4$, the resulting seed primes $q$ become massive roughly $q \approx N^4$.

\renewcommand{\arraystretch}{1.3}
\begin{longtable}{|c|p{6cm}|p{5cm}|c|}
\caption{Execution times and explicitly generated parameters for the Generalized Degree $p=5$ Cryptographic Prime Generation Chain.}
\label{tab:degree5_results} \\
\hline
\textbf{Bits} & \textbf{Trusted Seed Prime ($q$)} & \textbf{Constructed Prime ($N$)} & \textbf{Time (ms)} \\
\hline
32 & \texttt{2318879223369766908647165365\newline 703390383681} & \texttt{10376766241} & 8.17 \\
\hline
64 & \texttt{4472100928452149897742558073\newline 4272798601692194043299058118\newline 9456727315492379380481} & \texttt{38669664147672032641} & 16.79 \\
\hline
96 & \texttt{1625537387286824001820950409\newline 8570657067665848604448774092\newline 8491154260714395696121776189\newline 9317849462277533542358343002\newline 61361} & \texttt{16884637525328513551\newline 0273220281} & 66.12 \\
\hline
128 & \texttt{1677868795061963263287256716\newline 4443039212421324628105220322\newline 2823040168287511940786605648\newline 8544704985811761532295344444\newline 6856029202812457536585170449\newline 337489971012221} & \texttt{53818550955417216334\newline 4449011382424989511} & 37.17 \\\hline
\end{longtable}

\section{Concluding Remarks}\label{sec:10}
In this paper, we present a practical bottom-up approach to generalized algebraic primality proving. Rather than relying on complicated recursive searches, our method constructs primes directly, ensuring that the structural bound $q > N^{p/2}-1$ is satisfied by design. Additionally, by working strictly within monogenic pure prime degree number fields, the underlying arithmetic remains fast and deterministic, completely avoiding heavy computational overhead.\\
However, there is a natural limitation to keep in mind regarding the choice of degree $p$. Because our construction is based on the cyclotomic relation $\Phi_p(N) = k \cdot q$, the seed prime $q$ scales roughly on the order of $N^{p-1}$. If $p$ is chosen to be large, the required seed $q$ becomes larger than the target prime $N$. Since modern cryptography requires $N$ itself to be massive, managing a seed prime that is exponentially larger is computationally impractical. For this reason, our framework is most powerful and effective for small prime degrees, such as $p=3$ and $p=5$. In these optimal cases, the scaling is well-balanced, allowing us to quickly generate large and structurally secure cryptographic primes in a strict $\tilde{\mathcal{O}}(\log^2 N)$ time.\\
Finally, we wish to point out why we restricted the degree of extension to a prime $p$, rather than allowing a general integer $n$. This restriction is vital because composite degrees introduce structural complications that affect both the clarity and efficiency of the test. When $p$ is prime, factorization $x^p - 1 = (x - 1)\Phi_p(x)$ provides a simple and binary way to control the multiplicative order. If $n$ is composite, the factorization is split into several cyclotomic polynomials. For example, when $n=4$, we have $x^4 - 1 = (x - 1)(x + 1)(x^2 + 1)$. In such cases, the test would have to manage multiple intermediate divisors, ruining the efficiency of the verification step. Furthermore, composite degrees naturally introduce intermediate subfields. Elements can fall into these smaller subfields, causing the generalized unitary group $\mathcal{G}_N(d)$ to lose its strict cyclic structure. Restricting our framework to prime degrees ensures that no intermediate subfields exist, keeping the group structure perfectly rigid and the primality proof entirely deterministic.
\bibliographystyle{amsplain}
\bibliography{references}  
\end{document}